\documentclass[11pt, twoside, a4paper]{article}
\usepackage{amsmath}

\usepackage{amssymb}
\usepackage{amsthm}  
\usepackage{caption} 
\usepackage{bigints}
\usepackage{textcomp}    				
\usepackage[T1]{fontenc} 				
\usepackage{marvosym}    				
\usepackage[sc]{mathpazo}  	 			
 \usepackage[english]{babel} 
 \usepackage{float} 
 \usepackage{graphics} 
 \usepackage{graphicx} 
\usepackage[utf8]{inputenc}
\usepackage{mathtools}
\usepackage{authblk} 					
\usepackage[usenames]{xcolor}  			
\usepackage{lastpage} 					

\usepackage{enumerate}	 				

\usepackage{hyperref} 					

\usepackage{capt-of} 
\usepackage{calrsfs}
\usepackage{lipsum}  
\usepackage{epsfig}
\newcommand{\R}{\mathbf{R}}

\newtheorem{theorem}{Theorem}[section]

\newtheorem{lemma}[theorem]{Lemma}

\theoremstyle{definition}

\theoremstyle{definition}
\newtheorem{remark}[theorem]{Remark}
\theoremstyle{definition}

\numberwithin{equation}{section}
\numberwithin{table}{section}
\numberwithin{figure}{section}

\DeclareMathAlphabet{\pazocal}{OMS}{zplm}{m}{n}

\numberwithin{equation}{section}

\title{On a converse of Sturm's comparison theorem }
\author[1]{\textbf{Angelo B. Mingarelli}\footnote{Corresponding author. Email: angelo@math.carleton.ca}}

\affil[1]{School of Mathematics and Statistics, Carleton University, Ottawa, Canada}

\newcommand{\doi}[1]{\url{https://doi.org/#1}}
\newcommand{\MR}[1]{\href{https://www.ams.org/mathscinet-getitem?mr=#1}{MR#1}}
\newcommand{\ZB}[1]{\href{https://zbmath.org/?q=an:#1}{Zbl~#1}}

\makeatletter
\renewcommand{\maketitle}{\bgroup\setlength{\parindent}{0pt}

\vspace{1truecm}
\begin{center}{\vbox{\titlefont\@title}}\end{center}
\vspace{0.5truecm}
\begin{center}{\@author} \end{center}

\egroup
}

\renewcommand{\@fnsymbol}[1]{%
    \ifcase#1 \or {\,\Letter\!} \or\textasteriskcentered\or \textasteriskcentered\textasteriskcentered 
    \else\@ctrerr\fi}
\makeatother

\newcommand*{\titlefont}{\fontsize{18}{21.6}\selectfont\textbf}

\def\R{\mathbb{R}}

\begin{document}

\maketitle

\pagestyle{plain}

\begin{center}
\noindent
\begin{minipage}{0.85\textwidth}\parindent=15.5pt



{\small{
\noindent {\bf Abstract.}
We show that Sturm's classical comparison theorem (SCT)  on the interlacing of zeros of solutions of pairs of real second order two-term ordinary differential equations necessarily fails if the usual Sturmian-type conditions on the coefficients are violated.  We also show that the conditions on the coefficients are, in some sense, best possible. We give a necessary and sufficient condition for the failure of SCT and we note that its converse is false, and that the comparison theorem may still hold with very mild hypotheses on the coefficients.
\smallskip

\smallskip

}}

\end{minipage}
\end{center}
\section{Introduction}

In the sequel we will always assume that the coefficient functions $q_1, q_2 $ are piecewise continuous in $I=[a, b]$. Hereafter, by {\it solution} we  shall always mean a nontrivial (or non-identically zero) solution of the equation under consideration.
In its simplest form, Sturm's Comparison Theorem states that if for all $t\in I= [a, b]$.
\begin{equation}\label{eq0}
q_1(t)\leq q_2(t),
 \end{equation}
and $u(t)$ is a solution of the ordinary differential equation
\begin{equation}\label{eq1}
u^{\prime\prime} + q_1(t)\, u = 0
\end{equation}
satisfying 
\begin{equation}\label{eq2}
u(a)=u(b)=0,
\end{equation}
with consecutive zeros at $a, b$, then every solution of the differential equation
\begin{equation}\label{eq3}
v^{\prime\prime} + q_2(t)\, v = 0,
\end{equation}
must have at least one zero in $[a, b]$ (actually in $(a, b)$), [\cite{cs}, p. 2].

First, in Theorem~\ref{th1}, we show that SCT is best possible in the sense that there exists a potential functions $q_1, q_2$ such that $q_1(t)\leq q_2(t)$ for all $t\in [a,b]$ {\it except} on a subinterval, $I_\varepsilon$, of arbitrarily small length, $\varepsilon$,  and for which SCT fails on the larger interval $I$. Exploiting this idea will then lead us to the general result later. We motivate said theorem by showing that solutions of \eqref{eq4} fixed by the initial condition $v(0)=1,v^\prime(0) = \lambda$ have their zeros moving to the right as $\lambda$ increases.

Next, we state a converse of SCT and show that the failure of STC is equivalent to the disconjugacy of \eqref{eq3}. Finally, we show that SCT may actually hold for certain classes of functions with very mild restrictions. In general though, the converse of SCT is false. Given the inaugural nature of this study the assumption on the coefficients functions is for simplicity of exposition only. These results should adapt in a natural way to the case where $q_1,q_2$ are locally Lebesgue integrable on $I$ but we will not undertake this modification here.

\section{On the failure of Sturm's comparison theorem}

Let $v(t, \lambda)$ be a solution of \eqref{eq3} satisfying the initial conditions
\begin{equation}\label{ic}
v(0)=1,\quad v^\prime(0) = \lambda,
\end{equation}
where $\lambda\in \R$ and $[a, b]=[0,\pi]$, for simplicity, as a simple change of independent variable will transform the former into the latter. We know that this solution $v$ is continuously differentiable in both $t$ and the parameter $\lambda$ [\cite{ph}, Chapter 5.3]. We now examine the motion of its zeros as $\lambda$ increases. Let $t_0(\lambda)$ denote a typical zero of $v$ for a given $\lambda$.
Since $v(t_0(\lambda), \lambda)=0$, the implicit function theorem implies that
\begin{equation}\label{eq13}
\frac{dt_0(\lambda)}{d\lambda} = - \frac{v^\prime(t_0,\lambda)v_\lambda(t_0,\lambda)}{v^\prime(t_0,\lambda)^2}
\end{equation}
and therefore the motion of the zero $t_0(\lambda)$ depends on the sign of the numerator of the quantity on the right. Observe that both $v, v_\lambda$ satisfy \eqref{eq3} as a function of $t$. Indeed, using a Sturmian argument it is easy to conclude that $v^{\prime\prime}v_{\lambda} - v_\lambda^{\prime\prime}v=0$ which, after an integration over $[0, t_0]$, gives,
$$v^\prime(t_0,\lambda)v_\lambda(t_0,\lambda) = v^\prime(0,\lambda)v_\lambda(0,\lambda) - v_\lambda^\prime(0,\lambda)v(0,\lambda)$$
Using the integral form of the solution $v$ we note that $v_\lambda(0,\lambda)=0$ and $v_\lambda^\prime(0,\lambda)=1$. Since $v(0,\lambda)=1$, we get that 
$$v^\prime(t_0,\lambda)v_\lambda(t_0,\lambda) =-1$$ so that \eqref{eq13} is necessarily positive, which is equivalent to saying that the zero increases (or moves to the right) as $\lambda$ increases.
Indeed, for a given small $\varepsilon>0$, and small positive values of $\lambda$ there is a zero, $t_0(\lambda)$, of $v(t)$ in Theorem~\ref{th1} below. As $\lambda$ increases this zero must eventually leave $I$ through $\pi$, thus showing that $v(t, \lambda)>0$ for all $t\in I$ and all sufficiently large $\lambda$. This sort of reasoning led us to the following result.

\begin{theorem}\label{th1}
There is a constant $\varepsilon_0>0$ such that for every $0 < \varepsilon < \varepsilon_0$, there are potential functions $q_1, q_2$ satisfying \eqref{eq0} except only on an interval $I_{\varepsilon}\subset (a, b)$ of length $\varepsilon$, the result of which is that SCT fails on $I$.
\end{theorem}

\begin{proof} Without loss of generality we can assume that $[a,b]= [0,\pi]$ as can be readily verified by a simple linear change of independent variable. Let $0<\varepsilon <1$, and let $q_1(t)=1$ on $[0,\pi]$. Define 
\[q_2(t) = \left \{ \begin{array}{ll}
		1 ,	&	\mbox{if  \ \ $0 \leq  t < \pi-\varepsilon$}, \\
		(1-\varepsilon)^2 ,	&	\mbox{if  \ \ $\pi-\varepsilon \leq t \leq \pi$}. \\
		\end{array}
	\right. 
\]
Clearly $q_1(t) = q_2(t)$ except on an interval of length $\varepsilon$. We show that there exists initial conditions associated with a solution $v$ of \eqref{eq3} such that for a given solution $u$ of \eqref{eq1} satisfying \eqref{eq2}, $v(t) \neq 0$ on $[0,\pi]$. This shows that SCT fails in this case.
To this end let $u$ be a solution of \eqref{eq1} with consecutive zeros at $[0,\pi]$ and $u(t) > 0$ in $(0, \pi)$. 

Define a solution $v(t)$ of \eqref{eq3} by requiring that 
\begin{equation}\label{eq04} 
v(0)=1,\quad v^\prime(0)=\lambda.
\end{equation}
We show that there exists $\Lambda$ such that for all $\lambda > \Lambda$ the solution $v(t) \neq 0$ on $[0, \pi]$. (This is sometimes referred to as the {\it shooting method}.)

It is easy to see that our solution $v(t)$ is given by
\[v(t) = \left \{ \begin{array}{ll}
		\lambda\sin t + \cos t ,	&	\mbox{if  \ \ $0 \leq  t < \pi-\varepsilon$}, \\
		c_1 \sin ((1-\varepsilon)t) + c_2 \cos ((1-\varepsilon)t) &	\mbox{if  \ \ $\pi-\varepsilon \leq t \leq \pi$}. \\
		\end{array}
	\right. 
\]
for appropriate constants $c_1, c_2$, each one a function of $\lambda$ and  $\varepsilon$, chosen such that $v$ is continuously differentiable at $t=\pi-\varepsilon$. This leads to the system of equations
\begin{align*}
\lambda \sin (\pi - \varepsilon) + \cos (\pi - \varepsilon) &= c_1 \sin ((1-\varepsilon)(\pi - \varepsilon))+ c_2 \cos  (1-\varepsilon)(\pi - \varepsilon)) \\
\lambda \cos (\pi - \varepsilon) - \sin (\pi - \varepsilon) &= c_1 (1-\varepsilon)  \cos ((1-\varepsilon)(\pi - \varepsilon))\nonumber \\
&- c_2 (1-\varepsilon)  \sin  (1-\varepsilon)(\pi - \varepsilon)).
\end{align*}
in the unknowns, $c_1, c_2$, or, upon simplification,
\begin{align*}
\lambda \sin (\varepsilon) - \cos (\varepsilon) &= c_1 \sin ((1-\varepsilon)(\pi - \varepsilon))+ c_2 \cos  (1-\varepsilon)(\pi - \varepsilon)) \\
- \lambda \cos (\varepsilon) - \sin (\varepsilon) &= c_1 (1-\varepsilon)  \cos ((1-\varepsilon)(\pi - \varepsilon))\nonumber \\
 &- c_2 (1-\varepsilon)  \sin  (1-\varepsilon)(\pi - \varepsilon)).
\end{align*}
A lengthy though straightforward calculation gives,
\begin{align*}
c_1 & = \frac{\lambda}{2\varepsilon-2}\left ((\varepsilon-2)\cos(\varepsilon(\pi - \varepsilon)) - \varepsilon\cos(\varepsilon( \pi - \varepsilon+2))\right )\nonumber\\ 
& - \frac{1}{2\varepsilon-2}\left ((\varepsilon-2)\sin(\varepsilon(\pi - \varepsilon)) + \varepsilon\sin(\varepsilon( \pi - \varepsilon+2))\right )\nonumber
\end{align*}
and
\begin{align*}
c_2 & = \frac{\lambda }{2\varepsilon-2}\left ((\varepsilon-2)\sin(\varepsilon(\pi -\varepsilon)) - \varepsilon\sin(\varepsilon( \pi - \varepsilon+2))\right )\\
& + \frac{1}{2\varepsilon-2}\left ((\varepsilon-2)\cos(\varepsilon(\pi - \varepsilon)) + \varepsilon\cos(\varepsilon( \pi - \varepsilon+2))\right ).\nonumber
\end{align*}

As a result, an explicit form for $v$, is 
\begin{gather}
v(t) = \left \{ \begin{array}{ll|}
		\lambda\sin t + \cos t ,	&	\mbox{if  \ \ $0 \leq  t < \pi-\varepsilon$},\\
		 \lambda f(\varepsilon, t) + g(\varepsilon,t)   &	\mbox{if  \ \ $\pi-\varepsilon \leq t \leq \pi$}. \label{008} \\
		\end{array}
	\right. 
\end{gather}
where 
\begin{align}
(2\varepsilon -2) \,f(\varepsilon, t) & = (\varepsilon-2)\cos(\varepsilon^2-\varepsilon \pi)\sin ((1-\varepsilon)t) \nonumber \\
& - \varepsilon\cos (\varepsilon (\pi-\varepsilon +2) )\sin ((1-\varepsilon)t) \nonumber \\
& - \varepsilon\sin (\varepsilon (\pi-\varepsilon +2)) \cos ((1-\varepsilon)t) \nonumber \\
& -(\varepsilon-2)\sin(\varepsilon^2-\varepsilon \pi) \cos ((1-\varepsilon)t), \label{eq7}
\end{align}
and,
\begin{align*}
(2\varepsilon -2)\,g(\varepsilon, t) & = 
(\varepsilon-2)\cos(\varepsilon^2-\varepsilon \pi)\cos ((1-\varepsilon)t)  \nonumber \\
& + \varepsilon\cos (\varepsilon (\pi-\varepsilon +2) )\cos ((1-\varepsilon)t)\nonumber \\
& - \varepsilon\sin (\varepsilon (\pi-\varepsilon +2)) \sin ((1-\varepsilon)t) \nonumber \\
& +(\varepsilon-2)\sin(\varepsilon^2-\varepsilon \pi) \sin ((1-\varepsilon)t). 
\end{align*}
Observe that 
\begin{equation}\label{eq9}
|g(\varepsilon,t)| \leq \frac{3}{1-\varepsilon}, \quad \pi-\varepsilon \leq t \leq \pi.
\end{equation}
Next we show that, for our given $\varepsilon$, $f(\varepsilon,t) >0$ for $t \in [\pi-\varepsilon, \pi]$ and all sufficiently large $\lambda$. 
Using trigonometric identities in \eqref{eq7} we can rewrite $f$ in the form
\begin{align}
(2-2\varepsilon)\,f(\varepsilon,t) & = (2-\varepsilon) \sin (\varepsilon(\pi-\varepsilon) +t(1-\varepsilon)  ) + \varepsilon \sin (\varepsilon(2+\pi-\varepsilon)+t(1-\varepsilon))\nonumber \\
& = (2-\varepsilon) \sin \eta +\varepsilon \sin (\eta + 2\varepsilon) \label{eq10}
\end{align}
where, $\eta \equiv \varepsilon(\pi-\varepsilon) +t(1-\varepsilon) $ satisfies,
\begin{equation}\label{eq11} \pi -\varepsilon \leq\eta \leq \pi,
\end{equation}
since $ \pi-\varepsilon \leq t \leq \pi$ and $0<\varepsilon<1$. Using \eqref{eq10} we find,
\begin{align}
(2-2\varepsilon)\,f(\varepsilon,t) &= 2\sin \eta + 2\varepsilon^2 \left ( \frac{\sin(\eta+2\varepsilon)-\sin \eta}{2\varepsilon} \right )\nonumber \\
& = 2\sin\eta + 2\varepsilon^2 \cos \xi,\nonumber
\end{align}
for some $\xi \in [\eta, \eta + 2\varepsilon] \subset [\pi-\varepsilon, \pi+2\varepsilon].$ Since $\sin \eta \geq \sin \varepsilon$ on account of \eqref{eq11}, we find that, for $t \in [\pi-\varepsilon, \pi]$, 
\begin{equation}\label{eq12}
(2-2\varepsilon)\,f(\varepsilon,t) \geq 2\sin\varepsilon - 2\varepsilon^2 > 0
\end{equation}
provided $\varepsilon < \varepsilon_0$ where $\sin \varepsilon_0 = \varepsilon_0^2$, ($\varepsilon_0 \approx 0.87$).  Thus, for such a given $\varepsilon$, we can choose $\lambda$ so large in \eqref{eq04} that, on account of \eqref{eq12}, \eqref{eq9} and \eqref{008} we can make, $v(t) >0$ on both $[0,\pi-\varepsilon]$ and $[\pi-\varepsilon, \pi]$, and so on $[0,\pi]$. It follows that SCT fails on this interval.
\end{proof}

\section{On a converse of Sturm's comparison theorem}
We state a form of the converse of the usual SCT next.
\begin{remark}\label{remx}
{\bf Converse of SCT:} Given a solution $u$ of \eqref{eq1} satisfying $u(a)=u(b)=0$, ($u(t) \neq 0$ in $(a, b)$), and that every solution $v$ of \eqref{eq3} vanishes at least once in $I$, then $q_1(t) \leq q_2(t)$ on $I$ . 

The previous example shows that a converse of SCT may be addressed and could be true. In other words, for our choice of $q_1, q_2$ we showed that, by appealing to the contrapositive, that $q_1(t)>q_2(t)$ on $I_\varepsilon$ implies the existence of a solution $v$ of \eqref{eq3} without zeros on $I=[0,\pi]$.
\end{remark}

\begin{lemma}\label{lem2}
Let $a, b$ be consecutive zeros of a solution $u$ of \eqref{eq1}. Let $q_2$ satisfy \eqref{eq0} except on an interval $I \subset [a, b]$ in which $q_1(t) > q_2(t)$ (for all $t\in I$). Then any solution $v$ of \eqref{eq3} has at most one zero in $I$.
\end{lemma}
\begin{proof}
This is clear since if there were two zeros, $t_0< t_1$ in $I$, then Sturm's Comparison Theorem applied to the interval $[t_0, t_1]$ would imply that $u(t)$ must have a zero there, contrary to hypothesis.
\end{proof}

An equation of the form \eqref{eq3} is said to be {\it disconjugate} on $I$ \cite{wac}, [ \cite{ph}, p. 384] if every solution has at most one zero in $I$.
\begin{remark} It follows from Lemma~\ref{lem2} that \eqref{eq3} is disconjugate on a subinterval of $[a,b]$. This leads to the next result.
\end{remark}
\begin{lemma}\label{lem3} {\rm ([\cite{wac}, Theorem 1] )}
The equation \eqref{lem3} is disconjugate on I if it has a solution without zeros on I. For a compact or an open interval I this condition is also necessary.
\end{lemma}

\begin{theorem}\label{thm2}
Sturm's Comparison Theorem for equations \eqref{eq1}, \eqref{eq3} fails on $I$ if and only if \eqref{eq3} is disconjugate on $[a,b]$
\end{theorem}
\begin{proof}
Assume that SCT fails on $[a,b]$. Thus, given a solution $u$ of \eqref{eq1} with consecutive zeros at $a, b$, there is a solution $v$ of \eqref{eq3} with no zeros on $I$.    It now follows from Lemma~\ref{lem3} above that, since $I $ is compact, \eqref{eq3} is disconjugate on $I$.   Conversely, assume that \eqref{eq3} is disconjugate on $[a,b]$.  then by the Lemma~\ref{lem3}, there is a solution $v$ such that $v(t)\neq 0$ in $[a,b]$. Hence SCT fails on $[a,b]$.
\end{proof}
\begin{remark} 
Of course, the assumption that \eqref{eq1} has a solution with two consecutive zeros is implicit in the statement of SCT. The previous result emphasizes that in order for the hypotheses of SCT to be satisfied at all, \eqref{eq1} cannot be disconjugate (or else we can never find a solution $u$ with $u(a)= u(b)=0$). In particular, if we impose the condition \eqref{eq0} on a pair of equations \eqref{eq1}-\eqref{eq3} with the assumption that \eqref{eq3} is disconjugate then, either by Theorem~\ref{thm2} or by variational arguments (see [\cite{wac}, Theorem 7]), it can be shown that \eqref{eq1} must also be disconjugate, and so there is no solution (other than the trivial one, which is always excluded in this paper) that vanishes at two points. It follows that the question of the application of SCT cannot be asked in this case as the space of solutions satisfying \eqref{eq0} consists of one point only, the zero solution.
Since we showed directly that our solution $v(t)$ of \eqref{eq3} is positive in $I$, we get that Theorem~\ref{th1} is also consequence of Theorem~\ref{thm2}.
\end{remark}
The question now is, assuming that there are solutions $u$ of \eqref{eq1} vanishing at two consecutive points, $a, b$, to what extent is \eqref{eq0} necessary? Theorem~\ref{th1} shows that even if \eqref{eq0} is violated on an arbitrarily small interval, SCT may fail. The next result is not unexpected.
\begin{theorem}
Let $q_1(t) > q_2(t)$ for all $t \in I$. Then SCT fails, in the sense that if $u$ is a solution of \eqref{eq1} satisfying $u(a)=u(b)=0$, there is a solution $v$ of \eqref{eq3} with no zeros in $I$.
\end{theorem}
\begin{proof}
We know from lemma~\ref{lem2} that the hypothesis implies that every solution $v$ has at most one zero, so that \eqref{eq3} is disconjugate. Hence SCT fails by Theorem~\ref{thm2}.
\end{proof}
On the other hand, the next result appears to be new and is complementary to both Theorem~\ref{th1} and SCT.
\begin{theorem}\label{thm3}
There are functions $q_1, q_2$ such that \eqref{eq0} holds on some arbitrarily small subinterval $J \subset I$ and such that SCT holds.
\end{theorem}
\begin{proof}
Once again we may assume, without loss of generality, that $I=[0, \pi]$. This said, let $J=[0, \varepsilon]$ where $0 <\varepsilon < \pi$ is given but fixed. Now define $q_1(t)=1$, $q_2(t) = 1/\delta^2$ where $\delta > 0$ is chosen so that $\delta < \min \{ 1, 2\varepsilon/3\pi\}$. Next, define $u$ by $u(t) = \sin t$ so that $u(0)=u(\pi)=0$ and $v$ by $v(t) = \cos (t/\delta)$.
Clearly $u(t) \neq 0$ for all $t\in (0,\pi)$, while $v(t) = 0$ at two points, namely $t=\pi \delta/2, 3\pi\delta/2$. Sturm's separation theorem now implies that any other solution $v$ of \eqref{eq3} must have at least one zero in the interval $(\pi \delta/2,  3\pi\delta/2)$ and so in $(0, \pi)$. Hence, SCT holds on $(0, \pi)$.
\end{proof}

\begin{remark}
Theorem~\ref{thm3}  shows that \eqref{eq0} may be weakened considerably in some cases while Theorem~\ref{th1} indicates that \eqref{eq0} is close to being best possible for the validity of SCT, in general. 
\end{remark}
The previous result is a special case of the following more general result.
\begin{theorem}\label{thm3}
Given $q_1$ and a solution $u$ satisfying \eqref{eq2}, there is a function $q_2$ and an interval $J\in I$ (whose length may be made arbitrarily small) such that \eqref{eq0} is satisfied only on $J$ and SCT holds on $I$. \end{theorem}
\begin{proof} Choose $M$ so large that $||q_1||_\infty \leq M$ and the equation \eqref{eq3} with $q_2(t) = M$ has a solution with at least two zeros on an arbitrary but fixed subinterval $J$ of $I$. This is always possible as the number of zeros of any solution of $v^{\prime\prime}+Mv=0$ becomes unlimited as $M \to \infty$. Applying Sturm's separation theorem to the remaining solutions of \eqref{eq3} we get that every solution must have at least one zero in $J$ and so in $I$.
\end{proof}
{\bf NOTE:} Of course, the smaller the interval $J$, the larger $M$ must be in the preceding result. Thus, if we are free to choose $q_2$ in SCT then we can relax the fundamental requirement \eqref{eq0} considerably. In addition, Theorem~\ref{thm3} above gives a counterexample to the expected converse of SCT as described in Remark~\ref{remx}.

\end{document}